\documentclass [10pt, leqno]{article}
\usepackage {graphicx}
\usepackage{wrapfig}
\usepackage{amsmath,amssymb, amsthm}
\usepackage[colorlinks,linkcolor=blue,citecolor=blue,urlcolor=blue, linktocpage=true,dvips]{hyperref}

\numberwithin{equation}{section}

\newtheorem{theorem}{Theorem}

\newtheorem{lemma}[theorem]{Lemma}
\newtheorem{corollary}[theorem]{Corollary}

\newtheorem{remark}[theorem]{Remark}

\def\P{\mathbf{P}}

\def\Q{\mathbf{Q}}

\begin{document}

\title{Random index of codivisibility}
 \author{Jos\'{e} L. Fern\'{a}ndez and  Pablo Fern\'{a}ndez}

 \date{\today}

 \renewcommand{\thefootnote}{\fnsymbol{footnote}}
\renewcommand{\thefootnote}{\arabic{footnote}}

 \maketitle

 \begin{abstract}
{The index of codivisibility  of a set of integers is the size of its largest  subset with a common prime divisor. For large random samples of integers, the index of codivisibility is approximately normal.}\end{abstract}

\section{Coprimality of an $r$-tuple of integers}

For $r$-tuples of integers, there are two ``natural'' notions of \textit{coprimality}: the integers $a_1,\dots, a_r$ are \textit{mutually coprime} if $\gcd(a_1,\dots,a_r)=1$, and they are \textit{pairwise coprime} if $\gcd(a_i, a_j)=1$ for each $i\ne j$, $i,j=1,\dots,r$; which we abbreviate, respectively,   as $(a_1,\dots,a_r)\in \text{C}$ and $(a_1,\dots, a_r)\in \text{PC}$.

\smallskip
In this note, we are interested in the random behavior of these (and some other intermediate) notions of coprimality. For any given integer $n \ge 2$, let us denote by $X^{(n)}_1,
X^{(n)}_2, \ldots$ a sequence of independent random variables which are
uniformly distributed in $\{1, 2, \ldots, n\}$  and are defined in a
certain given probability space endowed with a probability $\P$.

\smallskip
Fix $r\ge 2$. Concerning mutual coprimality, we have
\begin{equation}
\label{eq:proportion of C}
\lim_{n \to \infty}\P\big(\big(X^{(n)}_1, \ldots, X^{(n)}_r\big) \in\text{C}\big)
=\frac{1}{\zeta(r)}\, ,
\end{equation}
that is, the probability of an $r$-tuple of integers being mutually coprime is asymptotically~$1/\zeta(r)$. The case $r=2$ is a classical result of Dirichlet, (see, for instance, Theorem 332 in~\cite{HW}), while the extension to $r>2$ can be traced back all the way back to E.~Ces\`{a}ro (\cite{Ce3}, page~293); see also, for instance, \cite{Ch}, \cite{HS} and \cite{Ny}.

\smallskip
For pairwise coprimality, we have
\begin{equation}\label{eq:proportion of PC}
\lim_{n \to \infty}\P\big(\big(X^{(n)}_1, \ldots, X^{(n)}_r\big) \in \text{PC}\big)
=\prod_{p} \Big(\Big(1-\frac{1}{p}\Big)^{r}+\frac{r}{p}\Big(1-\frac{1}{p}\Big)^{r-1}\Big):=T_r.
\end{equation}
(In this paper, $\prod_{p}$ or $\max_p$ means product or maximum running over all primes~$p$). This result was advanced by M. Schroeder, \cite{Schroeder}, and proved by L.~Toth, \cite{To2004}, and also by J.~Cai and E.~Bach, \cite{CB2001}.

For $r=2$, mutual and pairwise coprimality coincide, and $T_2={1}/{\zeta(2)}$. For $r \to \infty$, the probability of mutual coprimality tends to 1, while that of pairwise coprimality tends to 0; a mere reflection of the fact that pairwise coprimality is a more demanding notion that mutual coprimality.

\smallskip
Observe that pairwise coprimality of an $r$-tuple $(a_1,\dots, a_r)$ of integers means that, for any prime $p$, $p$~divides at most one the coordinates $a_j$, while mutual coprimality means that any prime~$p$ divides at most $r-1$ of them. 
It is enlightening to rewrite the limits~\eqref{eq:proportion of C} and \eqref{eq:proportion of PC} as
\begin{align}
\label{eq:proportion of Cbis}
\lim_{n \to \infty}&\P\big(\big(X^{(n)}_1, \ldots, X^{(n)}_r\big) \in\text{C}\big)=\prod_p \Big(1-\frac{1}{p^r}\Big)=\prod_p \P(\textsc{bin}(r,1/p)\le r-1),
\\
\label{eq:proportion of PCbis}
\lim_{n \to \infty}&\P\big(\big(X^{(n)}_1, \ldots, X^{(n)}_r\big) \in \text{PC}\big)
=\prod_p \P(\textsc{bin}(r,1/p)\le 1),
\end{align}
where $\textsc{bin}(r,1/p)$ denotes a binomial variable with number of repetitions $r$ and probability of success, \textit{sic}, $1/p$.

\smallskip
It is natural to consider the following notion of coprimality intermediate between mutual and pairwise coprimality: for fixed $2\le k\le r$, we will say that the integers $(a_1,\dots, a_r)$ are \textit{$k$-wise relatively prime} (or simply $k$-coprime, or $k\text{C}$) if \textit{any}~$k$ of them are relatively prime. 
Or alternatively, if each prime $p$ divides at most $k-1$ of them. The case $k=2$ is pairwise coprimality, while $k=r$ corresponds to mutual coprimality.


\smallskip
Recently, J. Hu (see Corollary 2 in \cite{Hu2}) has proved that
\begin{equation}\label{eq:proportion of kC}
\lim_{n \to \infty}\P\big(\big(X^{(n)}_1, \ldots, X^{(n)}_r\big) \in k\text{C}\big)
=\prod_{p} \P(\textsc{bin}(r,1/p)\le k-1)\, ;
\end{equation}
thus effectively  interpolating between \eqref{eq:proportion of Cbis} and \eqref{eq:proportion of PCbis}. See \cite{FF} for an alternative proof and some further developments.

Notice how \eqref{eq:proportion of Cbis},  \eqref{eq:proportion of PCbis}, and more generally \eqref{eq:proportion of kC} are manifestations of the asymptotic total independence of divisibility by primes.
\subsection{Index of codivisibility}

For each prime $p$, we denote the indicator of divisibility by $p$ by $i_p$, that is, for any  positive integer $a$, we write $i_p(a)=1$, if $p\mid a$, and  $i_p(a)=0$, if $p\nmid a$. For a $r$-tuple of integers $(a_1,  \ldots, a_r)$, we write $i_p(a_1, \ldots, a_r)=\sum_{j=1}^r i_p(a_j)$, which registers how many of those $a_j$  are divisible by $p$. Finally, the \textit{index of codivisibility},  $I_r(a_1, \ldots, a_r)$, of the $r$-tuple $(a_1,  \ldots, a_r)$ is given by
$$
I_r(a_1, \ldots, a_r)=\max_{p} \big[i_p(a_1, \ldots, a_r)\big]\, .
$$

Notice that $0 \le I_r \le r$, and that $I_r(a_1, \ldots, a_r)\le k$ means that $(a_1, \ldots, a_r)$ is $(k+1)C$.  Actually, $I_r(a_1, \ldots, a_r)=r$ says that $(a_1, \ldots, a_r)$ is not mutually coprime, while $I_r(a_1, \ldots, a_r)=1$ simply signifies that $(a_1, \ldots, a_r)$ is pairwise coprime. Observe that $I_r(a_1, \ldots, a_r)=0$ means that no prime divides any of the $a$, so that $(a_1,\ldots, a_r)=(1,1,\ldots, 1)$.

We introduce now the random variable $W^{(n)}_r$ given by
$$
W^{(n)}_r=I_r(X^{(n)}_1, \ldots, X^{(n)}_r)\,,
$$
which registers the index of codivisibility of a random sample of $r$ integers not exceeding $n$.

Observe that $$
\P(W^{(n)}_r=0)=\Big(\frac{1}{n}\Big)^r\, , \qquad
\P(W^{(n)}_r=r)=1- \Big(\frac{1}{n}\Big)^r \sum_{d=1}^n \mu(d) \Big\lfloor\frac{n}{d}\Big\rfloor^r,
$$
and that, for each $0 \le k \le r$, we may rewrite \eqref{eq:proportion of kC} as
$$
\lim_{n\to \infty}\P(W^{(n)}_r\le k)=\prod_{p} \P(\textsc{bin}(r,1/p)\le k)\,.
$$

\smallskip

For any integer $r\ge 2$ fixed, consider the distribution function
\begin{equation}
\label{eq:def of F_t(r)}
F_r(t)=\prod_{p} \P(\textsc{bin}(r,1/p)\le t).
\end{equation}
Observe that  $F_r(t)=0$ if $t< 1$, and $F_r(t)=1$ if $t\ge r$.

We denote by $W_r$ a random variable with distribution function $F_r$. The variable $W_r$ takes values on $\{0,1, 2, \ldots, r\}$, and $\P(W_r=0)=0$, $\P(W_r=1)=1/T_r$ and $\P(W_r=r)=1-1/\zeta(r)$. Observe that $W_r$ (informally) registers the index of codivisibility of a random $r$-tuple of integers (with no bound $n$ on the integers).

\section{Asymptotic normality of the random index of codivisibility}

As $r \to \infty$, the distribution of the random index of codivisibility $W_r$ is asymptotically normal. More precisely we shall prove that
\begin{theorem}
\label{theor:uniform_aproximation_binomial}
There are absolute constants $A, B>0$ so that for every $t \in \mathbb{R}$
\begin{equation}
\label{eq:uniform_aproximation_binomial}
0 \le \P\big(\textsc{bin}(r,{1}/{2})\le t\big)-\P\big(W_r \le t\big)\le A e^{-Br}\,.
\end{equation}
\end{theorem}
Thus, informally,  the index of codivisibility of a sequence of length $r$ ($r$~large) of  random numbers follows  (approximately) a binomial distribution  with $r$ repetitions and probability of success ${1}/{2}$. The Central Limit Theorem gives as an immediate consequence that:
\begin{corollary}\label{cor:normality of maximum coprimality}
$$
\frac{W_r-r/2}{\sqrt{r}/2} \overset{\text{d}}{\longrightarrow} \mathcal{N}(0,1)\,, \quad \text{as} \ \ r \to \infty\, .
$$
\end{corollary}

Observe that, since
$$
\P\big(W_r\le t\big)=\P\big(\textsc{bin}(r,1/2)\le t\big)\, \prod_{p\ge 3} \P\big(\textsc{bin}(r,1/p)\le t\big)\, ,
$$
we have, for every $t \in \mathbb{R}$,
\begin{equation*}
\P\big(\textsc{bin}(r,1/2)\le t\big)-\P\big(W_r\le t\big)=\P\big(\textsc{bin}(r,1/2)\le t\big)\Big(1-\prod_{p\ge 3} \P\big(\textsc{bin}(r,1/p)\le t\big)\Big)\, .
\end{equation*}
From this identity, the left hand side inequality of \eqref{eq:uniform_aproximation_binomial} follows; and it also follows that
\begin{equation}\label{eq:bound on difference of distributions}
\P\big(\textsc{bin}(r,1/2)\le t\big)-\P\big(W_r\le t\big)\le1-\prod_{p\ge 3} \P\big(\textsc{bin}(r,1/p)\le t\big)\, .
\end{equation}

We collect in the following two lemmas  some bounds on tails of binomial distributions that we need.
\begin{lemma}\label{lemma:bound Hoeffding on tails of binomials}
For any integer $N \ge 2$ and any probability $q \le {1}/{3}$,
$$
\P\big(\textsc{bin}(N, q) \le \tfrac{3}{8} N\big) \ge 1-e^{-N/300}\, .
$$
\end{lemma}
\begin{proof}
We bound
\begin{align}
\nonumber
\P\big(\textsc{bin}(N,q)\ge \tfrac{3}{8}N\big)&=\P\big(\textsc{bin}(N,q)-Nq\ge (\tfrac{3}{8}-q)N\big)\\&\le\P\big(\textsc{bin}(N,q)-Nq\ge \tfrac{N}{24}\big)\le e^{-2 \tfrac{N}{24^2}}\,.
\end{align}
The last inequality follows from the the standard Hoeffding's inequality (see, for instance, Theorem 2.1 in \cite{DL} or Theorem A.1.4 in \cite{AS}).
\end{proof}

\smallskip

\begin{lemma}\label{lemma:bound Bennet on tails of binomials}
For any integer $N \ge 2$ and any probability $q \le {1}/{64}$,
$$
\P\big(\textsc{bin}(N, q) \le \tfrac{3}{8} N\big) \ge 1-q^{3N/16}\, .
$$
\end{lemma}
\begin{proof}
We shall resort to Bennett's inequality (see, for instance, Exercise 2.5 in~\cite{DL} or Theorem A.1.12 in \cite{AS}) which  for any integer $N$, probability $q$, and $s >0$, gives the bound
\begin{equation}
\label{eq:bennets}
\P\big(\textsc{bin}(N,q) \ge Nq(1+s)\big) \le \exp\big(-Nq\, [(1+s)\ln(1+s)-s]\big)\, .
\end{equation}
In our case $Nq(1+s)=\frac{3}{8}N$. Since $q < {3}/{8}$, we have, as required, that $s >0$.
We bound the exponent in \eqref{eq:bennets} by
$$
\begin{aligned}
Nq\, [(1+s)\ln(1+s)-s]=\tfrac{3}{8}N\ln\big(\tfrac{3}{8q}\big)-Nqs\ge \tfrac{3}{8}N\ln\big(\tfrac{3}{8q}\big)-\tfrac{3}{8}N
\ge \tfrac{3}{8}N\ln\big(\tfrac{1}{8q}\big)\, ,
\end{aligned}$$
to conclude that
$$
\P\big(\textsc{bin}(N, q) \ge \tfrac{3}{8} N\big) \le (8q)^{{3N}/{8}}\le q^{3N/16}\, .
$$
We have used in the last inequality that $q\le{1}/{64}$.
\end{proof}

We are now ready for:
\begin{proof}[Proof of Theorem {\upshape\ref{theor:uniform_aproximation_binomial}}]
To prove the inequality of the theorem we may assume that $r$ is large, say $r\ge 16$. 
For $t \in \mathbb{R}$, write
$$
\Pi(t)=\P\big(\textsc{bin}(r,\tfrac{1}{2})\le t\big)-\P\big(W_r \le t\big)\, .
$$

To bound $\Pi(t)$, we split into two cases.
For  $t \le{3r}/{8}$, we simply bound
\begin{equation}\label{eq:bound for t small}
\begin{aligned}
\Pi(t)\le \P\big(\textsc{bin}(r,\tfrac{1}{2})\le t\big) &\le \P\big(\textsc{bin}(r,\tfrac{1}{2})\le \tfrac{3}{8}r\big)=
\P\big(\textsc{bin}(r,\tfrac{1}{2})-\tfrac{r}{2}\le -\tfrac{r}{8}\big)\\&=
\P\big(\textsc{bin}(r,\tfrac{1}{2})-\tfrac{r}{2}\ge \tfrac{r}{8}\big)\le e^{-{r}/{32}}\, .
\end{aligned}\end{equation}
In the last inequality above we have used again Hoeffding's inequality.

\smallskip

For $t >{3r}/{8}$, we first appeal to \eqref{eq:bound on difference of distributions}
to bound
$$
\Pi(t) \le 1-\prod_{p\ge 3} \P\big(\textsc{bin}(r,1/p)\le t\big) \le 1-\prod_{p\ge 3} \P\big(\textsc{bin}(r,1/p)\le \tfrac{3}{8}r\big)\, .
$$
Split now the product over primes into the product over $p \ge64$ and over $64 > p \ge 3$.
For the first product we have, using Lemma \ref{lemma:bound Bennet on tails of binomials},
\begin{equation}\label{eq:first bound for t large}
\prod_{p\ge 64} \P\big(\textsc{bin}(r,1/p)\le \tfrac{3}{8}r\big)\ge \prod_{p\ge 64} \Big(1-\frac{1}{p^{3r/16}}\Big)\ge \frac{1}{\zeta(3r/16)}\ge 1-2^{1-3 r/16}\, ,
\end{equation}
where we have used that $\zeta(s)\le 1+2^{1-s}$ for $s \ge 3$ (and that $3r/16 \ge 3$).
On the other hand, using Lemma \ref{lemma:bound Hoeffding on tails of binomials}, we may write
\begin{equation}\label{eq:second bound for t large}
\prod_{64>p \ge 3} \P\big(\textsc{bin}(r,1/p)\le \tfrac{3}{8}r\big)\ge \big(1-e^{-r/300}\big)^{17}\ge 1-17 e^{-r/300}\,.
\end{equation}
The proof is finished by taking into account the estimates \eqref{eq:bound for t small}, \eqref{eq:first bound for t large} and~\eqref{eq:second bound for t large}.
\end{proof}

\begin{remark}
{\upshape The argument of the proof of Theorem \ref{theor:uniform_aproximation_binomial} would give that if $q_1={1}/{2} >q_2 > q_3> \cdots >0$ is a sequence of probabilities so that $\sum_{j=1}q_j^\alpha <+\infty$, for some $\alpha >0$, and that if $\{U_j^{(N)}\}_{j\ge 1}$ is a sequence of independent binomial distributions, with parameters $N$ and $q_j$, then
the variable $V_N=\max_{j\ge1} (U^{(N)}_j)$ is asymptotically normal, and in fact, for each $t \in \mathbb{R}$,
$$
\lim_{N \to \infty}\P\Big(\frac{V_N-{N}/{2}}{{\sqrt{N}}/{2}}\le t\Big)=\Phi(t)\, .
$$}
\end{remark}

\begin{remark}
{\upshape  Let $X_1, X_2, \ldots$ be a sequence of independent variables all  following the zeta-distribution $\Q_s$ for some $s >1$: that is, for each integer $n \ge 1$,
$$
\Q_s(X=n)=\frac{1}{n^s\zeta(s)}.
$$
 Define, for $r\ge2$, $U_r=I_r(X_1, \ldots, X_r)$. Then,
 $$
\frac{U_r-{r}/{2^s}}{\sqrt{(1-{1}/{2^s})\, {r}/{2^s}}} \overset{\text{d}}{\longrightarrow} \mathcal{N}(0,1)\,, \quad \text{as} \ \ r \to \infty\, .
$$
This follows form the previous remark or by a slight modification of the proof of Theorem \ref{theor:uniform_aproximation_binomial}; just observe that divisibility by different primes are independent variables under $\Q_s$.}
\end{remark}

\smallskip

\noindent\textsc{Jos\'{e} L. Fern\'{a}ndez:} Departamento de Matem\'{a}ticas, Universidad Aut\'{o}noma de Madrid, 28049-Madrid, Spain.
\texttt{joseluis.fernandez@uam.es}.

\medskip

\noindent\textsc{Pablo Fern\'{a}ndez:} Departamento de Matem\'{a}ticas, Universidad Aut\'{o}noma de Madrid, 28049-Madrid, Spain.
\texttt{pablo.fernandez@uam.es}.

 \renewcommand{\thefootnote}{\fnsymbol{footnote}}
\footnotetext{The research of both authors is partially supported by Fundaci\'{o}n Akusmatika. The second named author is partially supported by the Spanish Ministerio de Ciencia e Innovaci\'{o}n, project no. MTM2011-22851.}
\renewcommand{\thefootnote}{\arabic{footnote}}

\end{document}